\documentclass{book}
\usepackage{amsmath,amssymb,amsthm}
\newtheorem{theorem}{Theorem}
\newtheorem{fact}{Fact}
\newtheorem{proposition}{Proposition}
\newtheorem{remark}{Remark}
\newtheorem{conjecture}{Conjecture}
\newtheorem{corollary}{Corollary}

\DeclareMathOperator{\Sym}{Sym}
\DeclareMathOperator{\Aut}{Aut}

\begin{document}

\chapter{The Random Graph}
\label{ch32:chap32}

%

\textbf{Summary.} Erd\H{o}s and R\'{e}nyi showed the paradoxical
result that there is a unique (and highly symmetric) countably
infinite random graph. This graph, and its automorphism group, form
the subject of the present survey.

\section{Introduction}
\label{ch32:sec2.1}

In 1963, Erd\H{o}s and R\'{e}nyi \cite{ch32:bib18} showed:
\begin{theorem}\label{ch32:them1.1} 
There exists a graph $R$ with the following
property. If a countable graph is chosen at random, by selecting
edges independently with probability $\frac{1}{2}$ from the set of
$2$-element subsets of the vertex set, then almost surely (i.e., with
probability $1$), the resulting graph is isomorphic to $R$.
\end{theorem}

This theorem, on first acquaintance, seems to defy common sense --- a
random process whose outcome is predictable. Nevertheless, the
argument which establishes it is quite short. (It is given below.)
Indeed, it formed a tailpiece to the paper of Erd\H{o}s and
R\'{e}nyi, which mainly concerned the much less predictable world of
finite random graphs. (In their book \emph{Probabilistic Methods in
Combinatorics}, Erd\H{o}s and Spencer \cite{ch32:bib19} remark that
this result ``demolishes the theory of infinite random graphs.'')

I will give the proof in detail, since it underlies much that
follows. The key is to consider the following property, which a
graph may or may not have:
\begin{itemize}
\item[($\ast$)] \emph{Given finitely many distinct vertices $u_1, \ldots, u_m, v_1, \ldots, v_n$, there exists a vertex $z$ which is adjacent to $u_1,\ldots, u_m$ and nonadjacent to $v_1,\ldots,v_n$.}
\end{itemize}
Often I will say, for brevity, ``$z$ is correctly joined''.
Obviously, a graph satisfying ($\ast$) is infinite, since $z$ is
distinct from all of $u_1, \ldots, u_m, v_1, \ldots, v_n$. It is
not obvious that any graph has this property. The theorem follows
from two facts:

\begin{fact}\label{ch32:subsec2.1.1}
With probability $1$, a countable random graph satisfies ($\ast$).
\end{fact}

\begin{fact}\label{ch32:subsec2.1.2} 
Any two countable graphs satisfying ($\ast$) are isomorphic.
\end{fact}

\begin{proof}[Proof (of Fact~\ref{ch32:subsec2.1.1})]
We have to show
that the event that ($\ast$) fails has probability $0$, i.e., the set
of graphs not satisfying ($\ast$) is a null set. For this, it is
enough to show that the set of graphs for which ($\ast$) fails for
some given vertices $u_1, \ldots, u_m, v_1, \ldots, v_n$ is null.
(For this deduction, we use an elementary lemma from measure theory:
the union of countably many null sets is null. There are only
countably many values of $m$ and $n$, and for each pair of values,
only countably many choices of the vertices $u_1, \ldots, u_m, v_1,
\ldots, v_n$.) Now we can calculate the probability of this set.
Let $z_1,\ldots,z_N$ be vertices distinct from $u_1, \ldots, u_m,
v_1, \ldots, v_n$. The probability that any $z_i$ is not correctly
joined is $1- \frac{1}{2^{m+n}}$; since these events are independent
(for different $z_i$), the probability that none of $z_1,\ldots,z_N$
is correctly joined is $(1 - \frac{1}{2^{m+n}})^N$. This tends to $0$
as $N \rightarrow \infty$; so the event that no vertex is correctly
joined does have probability $0$.
\end{proof}

Note that, at this stage, we know that graphs satisfying ($\ast$)
exist, though we have not constructed one --- a typical
``probabilistic existence proof''. Note also that ``probability
$\frac{1}{2}$'' is not essential to the proof; the same result holds
if edges are chosen with fixed probability $p$, where $0 < p < 1$.
Some variation in the edge probability can also be permitted.

\begin{proof}[Proof (of Fact~\ref{ch32:subsec2.1.2})]
Let $\Gamma_1$ and
$\Gamma_2$ be two countable graphs satisfying ($\ast$). Suppose that
$f$ is a map from a finite set $\{x_1,\ldots, x_n\}$ of vertices of
$\Gamma_1$ to $\Gamma_2$, which is an isomorphism of induced
subgraphs, and $x_{n+1}$ is another vertex of $\Gamma_1$. We show
that $f$ can be extended to $x_{n+1}$. Let $U$ be the set of
neighbours of $x_{n+1}$ within $\{x_1,\ldots, x_n\}$, and $V =
\{x_1,\ldots, x_n\} \setminus U$. A potential image of $x_{n+1}$
must be a vertex of $\Gamma_2$ adjacent to every vertex in 
$f(U)$ and nonadjacent to every vertex in $f(V)$. Now property
($\ast$) (for the graph $\Gamma_2$) guarantees that such a vertex
exists.

Now we use a model-theoretic device called ``back-and-forth''. (This
is often attributed to Cantor \cite{ch32:bib11}, in his
characterization of the rationals as countable dense ordered set
without endpoints. However, as Plotkin \cite{ch32:bib41} has shown,
it was not used by Cantor; it was discovered by Huntington
\cite{ch32:bib32} and popularized by Hausdorff \cite{ch32:bib25}.)

Enumerate the vertices of $\Gamma_1$ and $\Gamma_2$, as $\{x_1,
x_2, \ldots\}$ and $\{y_1, y_2, \ldots\}$ respectively. We build
finite isomorphisms $f_n$ as follows. Start with $f_0 = \emptyset$.
Suppose that $f_n$ has been constructed. If $n$ is even, let $m$ be
the smallest index of a vertex of $\Gamma_1$ not in the domain of
$f_n$; then extend $f_n$ (as above) to a map $f_{n+1}$ with $x_m$ in
its domain. (To avoid the use of the Axiom of Choice, select the
correctly-joined vertex of $\Gamma_2$ with smallest index to be the
image of $x_m$.) If $n$ is odd, we work backwards. Let $m$ be the
smallest index of a vertex of $\Gamma_2$ which is not in the range
of $f_n$; extend $f_n$ to a map $f_{n+1}$ with $y_m$ in its range
(using property ($\ast$) for $\Gamma_1$).

Take $f$ to be the union of all these partial maps. By going
alternately back and forth, we guaranteed that every vertex of
$\Gamma_1$ is in the domain, and every vertex of $\Gamma_2$ is in
the range, of $f$. So $f$ is the required isomorphism.
\end{proof}

The graph $R$ holds as central a position in graph theory as
$\mathbb{Q}$ does in the theory of ordered sets. It is surprising
that it was not discovered long before the 1960s! Since then, its
importance has grown rapidly, both in its own right, and as a prototype
for other theories.

\begin{remark}\rm
Results of Shelah and Spencer
\cite{ch32:bib47} and Hrushovski \cite{ch32:bib31} suggest that
there are interesting countable graphs which ``control'' the
first-order theory of finite random graphs whose edge-probabilities
tend to zero in specified ways. See Wagner \cite{ch32:bib54},
Winkler \cite{ch32:bib55} for surveys of this.
\end{remark}

\section{Some constructions}
\label{ch32:sec2.2}

Erd\H{o}s and R\'{e}nyi did not feel it necessary to give an
explicit construction of $R$; the fact that almost all countable
graphs are isomorphic to $R$ guarantees its existence. Nevertheless,
such constructions may tell us more about $R$. Of course, to show
that we have constructed $R$, it is necessary and sufficient to
verify condition ($\ast$).

I begin with an example from set theory. The downward
L\"{o}wenheim-Skolem theorem says that a consistent first-order
theory over a countable language has a countable model. In
particular, there is a countable model of set theory (the
\emph{Skolem paradox}).

\begin{theorem}\label{ch32:them2.1} 
Let M be a countable model of set theory.
Define a graph $M^{\ast}$ by the rule that $x\sim y$ if and only if
either $x\in y$ or $y\in x$. Then $M^{\ast}$ is isomorphic to $R$.
\end{theorem}
\begin{proof}
Let $u_1, \ldots, u_m, v_1, \ldots, v_n$ be
distinct elements of $M$. Let $x = \{v_1,\ldots, v_n\}$ and $z =
\{u_1, \ldots, u_m, x\}$. We claim that $z$ is a witness to
condition ($\ast$). Clearly $u_i\sim z$ for all $i$. Suppose that
$v_j \sim z$. If $v_j \in z$, then either $v_j = u_i$ (contrary to
assumption), or $v_j = x$ (whence $x \in x$, contradicting the Axiom
of Foundation). If $z\in v_j$, then $x \in z \in v_j \in x$, again
contradicting Foundation.
\end{proof}

Note how little set theory was actually used: only our ability to
gather finitely many elements into a set (a consequence of the Empty
Set, Pairing and Union Axioms) and the Axiom of Foundation. In
particular, the Axiom of Infinity is not required. Now there is a
familiar way to encode finite subsets of $\mathbb{N}$ as natural
numbers: the set $\{a_1, \ldots,a_n\}$ of distinct elements is
encoded as $2^{a_1}+ \cdots+ 2^{a_n}$. This leads to an explicit
description of $R$: the vertex set is $\mathbb{N}$; $x$ and $y$ are
adjacent if the $x^{\rm th}$ digit in the base $2$ expansion of $y$ is
a $1$ or \emph{vice versa}. This description was given by Rado
\cite{ch32:bib42}.

The next construction is more number-theoretic. Take as vertices the
set $\mathbb{P}$ of primes congruent to $1 \pmod{4}$. By quadratic
reciprocity, if $p,q \in\mathbb{P}$, then $\big(\frac{p}{q}\big) = 1$ if
and only if $\big(\frac{q}{p}\big) = 1$. (Here ``$\big(\frac{p}{q}\big) = 1$''
means that $p$ is a quadratic residue $\pmod{q}$.) We declare $p$ and $q$
adjacent if $\big(\frac{p}{q}\big)= 1$.

Let $u_1, \ldots, u_m, v_1, \ldots, v_n\in\mathbb{P}$. Choose a
fixed quadratic residue $a_i\pmod{u_i}$ (for example, $a_i = 1$),
and a fixed non-residue $b_j\pmod{v_j}$. By the Chinese Remainder
Theorem, the congruences
\[ 
x \equiv 1\pmod{4},\quad x \equiv a_i \pmod{u_i},\quad x \equiv b_j \pmod{v_j},
\]
have a unique solution $x \equiv x_0\pmod {4u_1 \ldots u_m v_1
\ldots  v_n}$. By Dirichlet's Theorem, there is a prime $z$
satisfying this congruence. So property ($\ast$) holds.

A set $S$ of positive integers is called $universal$ if, given
$k\in\mathbb{N}$ and $T \subseteq \{1,\ldots, k\}$, there is an
integer $N$ such that, for $i = 1, \ldots,k$,
\[
N+ i \in S\quad\mbox{if and only if}\quad i \in  T.
\]
(It is often convenient to consider binary sequences instead of
sets. There is an obvious bijection, under which the sequence
$\sigma$ and the set $S$ correspond when $(\sigma_i = 1)
~\Leftrightarrow~(i \in S)$ --- thus $\sigma$ is the characteristic
function of $S$. Now a binary sequence $\sigma$ is universal if and
only if it contains every finite binary sequence as a consecutive
subsequence.)

Let $S$ be a universal set. Define a graph with vertex set
$\mathbb{Z}$, in which $x$ and $y$ are adjacent if and only if 
$|x-y|\in S$. This graph is isomorphic to $R$. For let $u_1, \ldots,
u_m, v_1, \ldots, v_n$ be distinct integers; let $l$ and $L$ be the
least and greatest of these integers. Let $k = L - l + 1$ and $T =
\{u_i - l + 1 : i = 1, \ldots, m\}$. Choose $N$ as in the
definition of universality. Then $z = l - 1 - N$ has the required
adjacencies.

The simplest construction of a universal sequence is to enumerate
all finite binary sequences and concatenate them. But there are many
others. It is straight\-forward to show that a random subset of
$\mathbb{N}$ (obtained by choosing positive integers independently
with probability $\frac{1}{2}$) is almost surely universal. (Said
otherwise, the base $2$ expansion of almost every real number in 
$[0,1]$ is a universal sequence.)

Of course, it is possible to construct a graph satisfying ($\ast$)
directly. For example, let $\Gamma_0$ be the empty graph; if
$\Gamma_k$ has been constructed, let $\Gamma_{k+1}$ be obtained by
adding, for each subset $U$ of the vertex set of $\Gamma_k$, a
vertex $z(U)$ whose neighbour set is precisely $U$. Clearly, the
union of this sequence of graphs satisfies ($\ast$).

\section{Indestructibility}
\label{ch32:sec2.3}

The graph $R$ is remarkably stable: if small changes are made to it,
the resulting graph is still isomorphic to $R$. Some of these
results depend on the following analogue of property ($\ast$), which
appears stronger but is an immediate consequence of ($\ast$) itself.
\begin{proposition}\label{ch32:prop3.1} 
Let $u_1, \ldots, u_m, v_1, \ldots, v_n$ be
distinct vertices of $R$. Then the set
\[
Z = \{z : z\sim u_i\mbox{ for }i = 1,\ldots,m; z\nsim v_j\mbox{ for }j = 1,\ldots n\}
\]
is infinite; and the induced subgraph on this set is isomorphic to
$R$.
\end{proposition}
\begin{proof}
It is enough to verify property ($\ast$) for $Z$. So
let $u_1^{\prime}, \ldots, u_k^{\prime}, v_1^{\prime}, \ldots,
v_l^{\prime}$  be distinct vertices of $Z$. Now the vertex $z$
adjacent to $u_1, \ldots, u_n, u_1^{\prime}, \ldots, u_k^{\prime}$
and not to $v_1, \ldots, v_n, v_1^{\prime}, \ldots, v_l^{\prime}$,
belongs to $Z$ and witnesses the truth of this instance of ($\ast$)
there.
\end{proof}

The operation of $switching$ a graph with respect to a set $X$ of
vertices is defined as follows. Replace each edge between a vertex
of $X$ and a vertex of its complement by a non-edge, and each such
non-edge by an edge; leave the adjacencies within $X$ or outside $X$
unaltered. See Seidel \cite{ch32:bib46} for more properties of this
operation.
\begin{proposition}\label{ch32:prop3.2} 
The result of any of the following operations
on $R$ is isomorphic to $R$:
\begin{itemize}
\item[(a)] deleting a finite number of vertices;
\item[(b)] changing a finite number of edges into non-edges or vice versa;
\item[(c)] switching with respect to a finite set of vertices.
\end{itemize}
\end{proposition}
\begin{proof}
In cases (a) and (b), to verify an instance of
property ($\ast$), we use Proposition~\ref{ch32:prop3.1} to avoid
the vertices which have been tampered with. For (c), if $U
=\{u_1,\ldots,u_m\}$ and $V = \{v_1,\ldots, v_n\}$, we choose a
vertex outside $X$ which is adjacent (in $R$) to the vertices of $U
\setminus X$ and $V \cap X$, and non-adjacent to those of $U\cap X$
and $V \setminus X$.
\end{proof}

Not every graph obtained from $R$ by switching is isomorphic to $R$.
For example, if we switch with respect to the neighbours of a vertex
$x$, then $x$ is an isolated vertex in the resulting graph. However,
if $x$ is deleted, we obtain $R$ once again! Moreover, if we switch
with respect to a random set of vertices, the result is almost
certainly isomorphic to $R$.

$R$ satisfies the \emph{pigeonhole principle}:
\begin{proposition}\label{ch32:prop3.3} 
If the vertex set of $R$ is partitioned into a
finite number of parts, then the induced subgraph on one of these
parts is isomorphic to $R$.
\end{proposition}
\begin{proof}
Suppose that the conclusion is false for the
partition $X_1 \cup \ldots \cup X_k$ of the vertex set. Then, for
each $i$, property ($\ast$) fails in $X_i$, so there are finite
disjoint subsets $U_i, V_i$ of $X_i$ such that no vertex of $X_i$
is ``correctly joined'' to all vertices of $U_i$, and to none of
$V_i$. Setting $U = U_1 \cup \ldots \cup U_k$ and $V = V_1 \cup
\ldots \cup V_k$, we find that condition ($\ast$) fails in $R$ for
the sets $U$ and $V$, a contradiction.
\end{proof}

Indeed, this property is characteristic:
\begin{proposition}\label{ch32:prop3.4} 
The only countable graphs $\Gamma$ which have
the property that, if the vertex set is partitioned into two parts,
then one of those parts induces a subgraph isomorphic to $\Gamma$,
are the complete and null graphs and $R$.
\end{proposition}
\begin{proof}
Suppose that $\Gamma$ has this property but is not
complete or null. Since any graph can be partitioned into a null
graph and a graph with no isolated vertices, we see that $\Gamma$
has no isolated vertices. Similarly, it has no vertices joined to
all others.

Now suppose that $\Gamma$ is not isomorphic to $R$. Then we can find
$u_1,\ldots,u_m$ and $v_1,\ldots,v_n$ such that ($\ast$) fails, with
$m +n$ minimal subject to this. By the preceding paragraph, $m + n >
1$. So the set $\{u_1,\ldots, v_n\}$ can be partitioned into two
non-empty subsets $A$ and $B$. Now let $X$ consist of $A$ together
with all vertices (not in $B$) which are not ``correctly joined'' to
the vertices in $A$; and let $Y$ consist of $B$ together with all
vertices (not in $X$) which are not ``correctly joined'' to the
vertices in $B$. By assumption, $X$ and $Y$ form a partition of the
vertex set. Moreover, the induced subgraphs on $X$ and $Y$ fail
instances of condition ($\ast$) with fewer than $m +n$ vertices; by
minimality, neither is isomorphic to $\Gamma$, a contradiction.
\end{proof}

Finally:
\begin{proposition}\label{ch32:prop3.5} 
$R$ is isomorphic to its complement.
\end{proposition}

For property ($\ast$) is clearly self-complementary.

\section{Graph-theoretic properties}
\label{ch32:sec2.4}

The most important property of $R$ (and the reason for Rado's
interest) is that it is \emph{universal}: 

\begin{proposition}\label{ch32:prop4.1} 
Every finite or countable graph can be embedded
as an induced subgraph of $R$.
\end{proposition}
\begin{proof}
We apply the proof technique of Fact~\ref{ch32:subsec2.1.2}; 
but, instead of back-and-forth, we
just ``go forth''. Let $\Gamma$ have vertex set $\{x_1,
x_2,\ldots\}$, and suppose that we have a map $f_n : \{x_1,
\ldots,x_n\} \rightarrow R$ which is an isomorphism of induced
subgraphs. Let $U$ and $V$ be the sets of neighbours and
non-neighbours respectively of $x_{n+1}$ in $\{x_1, \ldots,x_n\}$.
Choose $z \in R$ adjacent to the vertices of $f(U)$ and nonadjacent
to those of $f(V)$, and extend $f_n$ to map $x_{n+1}$ to $z$. The
resulting map $f_{n+1}$ is still an isomorphism of induced
subgraphs. Then $f = \bigcup f_n$ is the required embedding. (The
point is that, going forth, we only require that property ($\ast$)
holds in the target graph.)
\end{proof}

In particular, $R$ contains infinite cliques and cocliques. Clearly
no finite clique or coclique can be maximal. There do exist infinite
maximal cliques and cocliques. For example, if we enumerate the
vertices of $R$ as $\{x_1, x_2, \ldots\}$, and build a set $S$ by
$S_0 = \emptyset, S_{n +1} = S_n \cup \{x_m\}$ where $m$ is the
least index of a vertex joined to every vertex in $S_n$, and $S =
\bigcup S_n$, then $S$ is a maximal clique.

Dual to the concept of induced subgraph is that of \emph{spanning
subgraph}, using all the vertices and some of the edges. Not every
countable graph is a spanning subgraph of $R$ (for example, the
complete graph is not). We have the following characterization:

\begin{proposition}\label{ch32:prop4.2} 
A countable graph $\Gamma$ is isomorphic to a
spanning subgraph of $R$ if and only if, given any finite set
$\{v_1, \ldots, v_n\}$ of vertices of $\Gamma$, there is a vertex
$z$ joined to none of $v_1, \ldots, v_n$.
\end{proposition}
\begin{proof}
We use back-and-forth to construct a bijection
between the vertex sets of $\Gamma$ and $R$, but when going back
from $R$ to $\Gamma$, we only require that \emph{nonadjacencies}
should be preserved.
\end{proof}

This shows, in particular, that every infinite locally finite graph
is a spanning subgraph (so $R$ contains $1$-factors, one- and two-way
infinite Hamiltonian paths, etc.). But more can be said.

The argument can be modified to show that, given any non-null
locally finite graph $\Gamma$, any edge of $R$ lies in a spanning
subgraph isomorphic to $\Gamma$. Moreover, as in the last section,
if the edges of a locally finite graph are deleted from $R$, the
result is still isomorphic to $R$. Now let $\Gamma_1,
\Gamma_2,\ldots$ be given non-null locally finite countable graphs.
Enumerate the edges of $R$, as $\{e_1,e_2,\ldots\}$. Suppose that
we have found edge-disjoint spanning subgraphs of $R$ isomorphic to
$\Gamma_1,\ldots,\Gamma_n$. Let $m$ be the smallest index of an edge
of $R$ lying in none of these subgraphs. Then we can find a spanning
subgraph of $R - (\Gamma_1 \cup\cdots\cup\Gamma_n)$ containing $e_m$
and isomorphic to $\Gamma_{n +1}$. We conclude:

\begin{proposition}\label{ch32:prop4.3} 
The edge set of $R$ can be partitioned into
spanning subgraphs isomorphic to any given countable sequence of
non-null countable locally finite graphs.
\end{proposition}

In particular, $R$ has a $1$-factorization, and a partition into
Hamiltonian paths.

\section{Homogeneity and categoricity}
\label{ch32:sec2.5}

We come now to two model-theoretic properties of $R$. These
illustrate two important general theorems, the
Engeler--Ryll-Nardzewski--Svenonius theorem and Fra\"{i}ss\'{e}'s
theorem. The context is first-order logic; so a \emph{structure}
is a set equipped with a collection of relations, functions and
constants whose names are specified in the language. If there are no
functions or constants, we have a \emph{relational structure}. The
significance is that any subset of a relational structure carries an
induced substructure. (In general, a substructure must contain the
constants and be closed with respect to the functions.)

Let $M$ be a relational structure. We say that $M$ is
\emph{homogeneous} if every isomorphism between finite induced
substructures of $M$ can be extended to an automorphism of $M$.

\begin{proposition}\label{ch32:prop5.1} 
$R$ is homogeneous.
\end{proposition}
\begin{proof}
In the proof of Fact~\ref{ch32:subsec2.1.2}, the
back-and-forth machine can be started with any given isomorphism
between finite substructures of the graphs $\Gamma_1$ and
$\Gamma_2$, and extends it to an isomorphism between the two
structures. Now, taking $\Gamma_1$ and $\Gamma_2$ to be $R$ gives
the conclusion.
\end{proof}

Fra\"{i}ss\'{e} \cite{ch32:bib21} observed that $\mathbb{Q}$ (as an
ordered set) is homogeneous, and used this as a prototype: he gave a
necessary and sufficient condition for the existence of a
homogeneous structure with prescribed finite substructures.
Following his terminology, the \emph{age} of a structure $M$ is the
class of all finite structures embeddable in $M$. A class
$\mathcal{C}$ of finite structures has the \emph{amalgamation
property} if, given $A$, $B_1$, $B_2 \in \mathcal{C}$ and embeddings
$f_1 : A \rightarrow B_1$ and $f_2 : A \rightarrow B_2$, there
exists $C \in\mathcal{C}$ and embeddings $g_1 : B_1 \rightarrow C$
and $g_2 : B_2 \rightarrow C$ such that $f_1g_1 = f_2g_2$. (Less
formally, if the two structures $B_1, B_2$ have isomorphic
substructures $A$, they can be ``glued together'' so that the copies
of $A$ coincide, the resulting structure $C$ also belonging to the
class $\mathcal{C}$.) We allow $A=\emptyset$ here.

\begin{theorem}\label{ch32:them5.1} 
\begin{itemize}
\item[(a)] A class $C$ of finite structures (over a
fixed relational language)is the age of a countable homogeneous
structure $M$ if and only if $\mathcal{C}$ is closed under
isomorphism, closed under taking induced substructures, contains
only countably many non-isomorphic structures, and has the
amalgamation property.
\item[(b)] If the conditions of (a) are satisfied,
then the structure $M$ is unique up to isomorphism.
\end{itemize}
\end{theorem}

A class $\mathcal{C}$ having the properties of this theorem is called
a \emph{Fra\"{\i}ss\'e class}, and the countable homogeneous structure
$M$ whose age is $\mathcal{C}$ is its \emph{Fra\"{i}ss\'{e} limit}. The
class of all finite graphs is Fra\"{i}ss\'{e} class;
its Fra\"{i}ss\'{e} limit is $R$. The Fra\"{i}ss\'{e}
limit of a class $\mathcal{C}$ is characterized by a condition
generalizing property ($\ast$): 
\emph{If $A$ and $B$ are members of the age of $M$ with $A\subseteq B$
and $|B| = |A| + 1$, then every embedding of $A$ into
$M$ can be extended to an embedding of $B$ into $M$.}

In the statement of the amalgamation property, when the two
structures $B_1, B_2$ are ``glued together'', the overlap may be
larger than $A$. We say that the class $\mathcal{C}$ has the
\emph{strong amalgamation property} if this doesn't occur;
formally, if the embeddings $g_1, g_2$ can be chosen so that, if
$b_1g_1 = b_2g_2$, then there exists $a \in A$ such that $b_1 =
af_1$ and $b_2 = af_2$. This property is equivalent to others we
have met.

\begin{proposition}\label{ch32:prop5.2} 
Let $M$ be the Fra\"{i}ss\'{e} limit of the
class $\mathcal{C}$, and $G = \Aut(M)$. Then the following are
equivalent:
\begin{itemize}
\item[(a)] $\mathcal{C}$ has the strong amalgamation property;
\item[(b)] $M \setminus A \cong M$ for any finite subset $A$ of $M$;
\item[(c)] the orbits of $G_A$ on $M \setminus A$ are infinite for any finite subset $A$ of $M$, where $G_A$
is the setwise stabiliser of $A$.
\end{itemize}
\end{proposition}

See Cameron \cite{ch32:bib6}, El-Zahar and Sauer \cite{ch32:bib15}.

A structure $M$ is called \emph{$\aleph_0$-categorical} if any countable
structure satisfying the same first-order sentences as $M$ is
isomorphic to $M$. (We must specify countability here: the upward
L\"{o}wenheim--Skolem theorem shows that, if $M$ is infinite, then
there are structures of arbitrarily large cardinality which satisfy
the same first-order sentences as $M$.)

\begin{proposition}\label{ch32:prop5.3} 
$R$ is $\aleph_0$-categorical.
\end{proposition}
\begin{proof}
Property ($\ast$) is not first-order as it stands,
but it can be translated into a countable set of first-order
sentences $\sigma_{m,n}$ (for $m, n \in \mathbb{N}$), where
$\sigma_{m,n}$ is the sentence
\[
(\forall u_1..u_mv_1..v_n)\bigg(\!\bigg({(u_1\neq
v_1)\& \ldots \& \atop (u_m \neq v_n)}\bigg)\rightarrow
(\exists z)\bigg({(z\sim u_1)\& \ldots \& (z\sim
u_m)\& \atop \neg(z\sim v_1)\& \ldots \& \neg(z\sim
v_n)}\bigg)\!\bigg).\qedhere
\]
\end{proof}

Once again this is an instance of a more general result. An
\emph{$n$-type} in a structure $M$ is an equivalence class of
$n$-tuples, where two tuples are equivalent if they satisfy the same
($n$-variable) first-order formulae. Now the following theorem was
proved by Engeler \cite{ch32:bib16}, Ryll-Nardzewski
\cite{ch32:bib44} and Svenonius \cite{ch32:bib49}:

\begin{theorem}\label{ch32:them5.2} 
For a countable first-order structure $M$, the
following conditions are equivalent:
\begin{itemize}
\item[(a)] $M$ is $\aleph_0$-categorical;
\item[(b)] $M$ has only finitely many $n$-types, for every $n$;
\item[(c)] the automorphism group of $M$ has only finitely many orbits on $M^n$, for every
$n$.
\end{itemize}
\end{theorem}

Note that the equivalence of conditions (a) (axiomatizability) and
(c) (symmetry) is in the spirit of Klein's Erlanger Programm. The
fact that $R$ satisfies (c) is a consequence of its homogeneity,
since $(x_1,\ldots, x_n)$ and $(y_1,\ldots,y_n)$ lie in the same
orbit of Aut($R$) if and only if the map $(x_i \rightarrow y_i)$ $(i
= 1,\ldots, n)$ is an isomorphism of induced subgraphs, and there
are only finitely many $n$-vertex graphs.

\begin{remark}\label{ch32:rem5.1}\rm 
The general definition of an $n$-type in
first-order logic is more complicated than the one given here:
roughly, it is a maximal set of $n$-variable formulae consistent
with a given theory. I have used the fact that, in an
$\aleph_0$-categorical structure, any $n$-type is $realized$ (i.e.,
satisfied by some tuple) --- this is a consequence of the
G\"{o}del--Henkin completeness theorem and the downward
L\"{o}wenheim--Skolem theorem. See Hodges \cite{ch32:bib28} for more
details.
\end{remark}

Some properties of $R$ can be deduced from either its homogeneity or
its $\aleph_0$-categoricity. For example,
Proposition~\ref{ch32:prop4.1} generalizes. We say that a countable
relational structure $M$ is \emph{universal} (or \emph{rich for its
age}, in Fra\"{i}ss\'{e}'s terminology \cite{ch32:bib22}) if every
countable structure $N$ whose age is contained in that of $M$ (i.e.,
which is \emph{younger} than $M$) is embeddable in $M$.

\begin{theorem}\label{ch32:them5.3} 
If $M$ is either $\aleph_0$-categorical or homogeneous, then it is universal.
\end{theorem}

The proof for homogeneous structures follows that of
Proposition~\ref{ch32:prop4.1}, using the analogue of property
($\ast$) described above. The argument for $\aleph_0$-categorical
structures is a bit more subtle, using Theorem 5.4 and K\"{o}nig's
Infinity Lemma: see Cameron \cite{ch32:bib7}.

\section{First-order theory of random graphs}
\label{ch32:sec2.6}

The graph $R$ controls the first-order theory of finite random
graphs, in a manner I now describe. This theory is due to Glebskii
{\it et al.} \cite{ch32:bib24}, Fagin \cite{ch32:bib20}, and
Blass and Harary \cite{ch32:bib2}. A property P holds in
\emph{almost all finite random graphs} if the proportion of
$N$-vertex graphs which satisfy P tends to $1$ as $N \rightarrow
\infty$. Recall the sentences $\sigma_{m,n}$ which axiomatize $R$.

\begin{theorem}\label{ch32:them6.1} 
Let $\theta$ be a first-order sentence in the
language of graph theory. Then the following are equivalent:
\begin{itemize}
\item[(a)] $\theta$ holds in almost all finite random graphs;
\item[(b)] $\theta$ holds in the graph $R$;
\item[(c)] $\theta$ is a logical consequence of $\{\sigma_{m,n} : m, n \in \mathbb{N}\}$.
\end{itemize}
\end{theorem}
\begin{proof}
The equivalence of (b) and (c) is immediate from the
G\"{o}del--Henkin completeness theorem for first-order logic and the
fact that the sentences $\sigma_{m,n}$ axiomatize $R$.

We show that (c) implies (a). First we show that $\sigma_{m,n}$
holds in almost all finite random graphs. The probability that it
fails in an $N$-vertex graph is not greater than $N^{m+n}(1 -
\frac{1}{2^{m+n}})^{N - m- n}$, since there are at most $N^{ m +n}$
ways of choosing $m + n$ distinct points, and $(1 -
\frac{1}{2^{m+n}})^{N - m- n}$ is the probability that no further
point is correctly joined. This probability tends to $0$ as $N
\rightarrow \infty$.

Now let $\theta$ be an arbitrary sentence satisfying (c). Since
proofs in first-order logic are finite, the deduction of $\theta$
involves only a finite set $\Sigma$ of sentences $\sigma_{m,n}$. It
follows from the last paragraph that almost all finite graphs
satisfy the sentences in $\Sigma$; so almost all satisfy $\theta$
too.

Finally, we show that not (c) implies not (a). If (c) fails, then
$\theta$ doesn't hold in $R$, so $(\neg \theta)$ holds in $R$, so
$(\neg \theta)$ is a logical consequence of the sentences
$\sigma_{m,n}$. By the preceding paragraph, $(\neg \theta)$ holds in
almost all random graphs.
\end{proof}

The last part of the argument shows that there is a zero-one law:

\begin{corollary}\label{ch32:coro6.1} 
Let $\theta$ be a sentence in the language of
graph theory. Then either $\theta$ holds in almost all finite random
graphs, or it holds in almost none.
\end{corollary}

It should be stressed that, striking though this result is, most
interesting graph properties (connectedness, hamiltonicity, etc.)
are not first-order, and most interesting results on finite random
graphs are obtained by letting the probability of an edge tend to
zero in a specified manner as $N \rightarrow \infty$, rather than
keeping it constant (see Bollob\'{a}s \cite{ch32:bib3}).
Nevertheless, we will see a recent application of
Theorem~\ref{ch32:them6.1} later.

\section{Measure and category}
\label{ch32:sec2.7}

When the existence of an infinite object can be proved by a
probabilistic argument (as we did with $R$ in
Section~\ref{ch32:sec2.1}), it is often the case that an
alternative argument using the concept of Baire category can be
found. In this section, I will sketch the tools briefly. See Oxtoby
\cite{ch32:bib39} for a discussion of measure and Baire category.

In a topological space, a set is \emph{dense} if it meets every nonempty
open set; a set is \emph{residual} if it contains a countable
intersection of open dense sets. The \emph{Baire category theorem}
states:

\begin{theorem}\label{ch32:them7.1} 
In a complete metric space, any residual set is non-empty.
\end{theorem}

(The analogous statement for probability is that a set which
contains a countable intersection of sets of measure $1$ is non-empty.
We used this to prove Fact~\ref{ch32:subsec2.1.1}.)

The simplest situation concerns the space $2^{\mathbb{N}}$ of all
infinite sequences of zeros and ones. This is a probability space,
with the ``coin-tossing measure'' --- this was the basis of our
earlier discussion --- and also a complete metric space, where we
define $d(x, y) = \frac{1}{2^n}$ if the sequences $x$ and $y$ agree
in positions $0, 1, \ldots, n - 1$ and disagree in position $n$.
Now the topological concepts translate into combinatorial ones as
follows. A set $S$ of sequences is open if and only if it is
\emph{finitely determined}, i.e., any $x \in S$ has a finite
initial segment such that all sequences with this initial segment
are in $S$. A set $S$ is dense if and only if it is \emph{always
reachable}, i.e., any finite sequence has a continuation lying in
$S$. Now it is a simple exercise to prove the Baire category theorem
for this space, and indeed to show that a residual set is dense and
has cardinality $2^{\aleph_0}$. We will say that ``almost all
sequences have property P (in the sense of Baire category)'' if the
set of sequences which have property P is residual.

We can describe countable graphs by binary sequences: take a fixed
enumeration of the $2$-element sets of vertices, and regard the
sequence as the characteristic function of the edge set of the
graph. This gives meaning to the phrase ``almost all graphs (in the
sense of Baire category)''. Now, by analogy with
Fact~\ref{ch32:subsec2.1.1}, we have:

\begin{fact}\label{ch32:subsec2.7.1} 
Almost all countable graphs (in the
sense of either measure or Baire category) have property ($\ast$).
\end{fact}

The proof is an easy exercise. In fact, it is simpler for Baire
category than for measure --- no limit is required!

In the same way, almost all binary sequences (in either sense) are
universal (as defined in Section~\ref{ch32:sec2.2}).

A binary sequence defines a path in the binary tree of countable
height, if we start at the root and interpret $0$ and $1$ as instructions
to take the left or right branch at any node. More generally, given
any countable tree, the set of paths is a complete metric space,
where we define the distance between two paths to be $\frac{1}{2^n}$
if they first split apart at level $n$ in the tree. So the concept
of Baire category is applicable. The combinatorial interpretation of
open and dense sets is similar to that given for the binary case.

For example, the age of a countable relational structure $M$ can be
described by a tree: nodes at level $n$ are structures in the age
which have point set $\{0, 1, \ldots,  n-1\}$, and nodes $X_n,
X_{n +1}$ at levels $n$ and $n +1$ are declared to be adjacent if
the induced structure of $X_{n+1}$ on the set $\{0, 1, \ldots,
n-1\}$ is $X_n$. A path in this tree uniquely describes a structure
$N$ on the natural numbers which is younger than $M$, and
conversely. Now Fact~\ref{ch32:subsec2.7.1} generalizes as follows:

\begin{proposition}\label{ch32:prop7.1} 
If $M$ is a countable homogeneous relational
structure, then almost all countable structures younger than $M$ are
isomorphic to $M$.
\end{proposition}

It is possible to formulate analogous concepts in the
measure-theoretic framework, though with more difficulty. But the
results are not so straightforward. For example, almost all finite
triangle-free graphs are bipartite (a result of Erd\H{o}s, Kleitman
and Rothschild \cite{ch32:bib17}); so the ``random countable
triangle-free graph'' is almost surely bipartite. (In fact, it is
almost surely isomorphic to the ``random countable bipartite graph'',
obtained by taking two disjoint countable sets and selecting edges
between them at random.)

A structure which satisfies the conclusion of
Proposition~\ref{ch32:prop7.1} is called \emph{ubiquitous} (or
sometimes \emph{ubiquitous in category}, if we want to distinguish
measure-theoretic or other forms of ubiquity). Thus the random graph
is ubiquitous in both measure and category. See Bankston and
Ruitenberg \cite{ch32:bib1} for further discussion.

\section{The automorphism group}
\label{ch32:sec2.8}

\subsection{General properties}

From the homogeneity of $R$ (Proposition~\ref{ch32:prop5.1}), we
see that it has a large and rich group of automorphisms: the
automorphism group $G = \Aut(R)$ acts transitively on the
vertices, edges, non-edges, etc. --- indeed, on finite
configurations of any given isomorphism type. In the language of
permutation groups, it is a rank $3$ permutation group on the vertex
set, since it has three orbits on ordered pairs of vertices, viz.,
equal, adjacent and non-adjacent pairs. Much more is known about
$G$; this section will be the longest so far.

First, the cardinality: 

\begin{proposition}\label{ch32:prop8.1} 
$|\Aut(R)| = 2^{\aleph_0}$.
\end{proposition}

This is a special case of a more general fact. The automorphism
group of any countable first-order structure is either at most
countable or of cardinality $2^{\aleph_0}$, the first alternative
holding if and only if the stabilizer of some finite tuple of points
is the identity.

The normal subgroup structure was settled by Truss \cite{ch32:bib51}:

\begin{theorem}\label{ch32:them8.1} 
$\Aut(R)$ is simple.
\end{theorem}

Truss proved a stronger result: if $g$ and $h$ are two non-identity
elements of $\Aut(R)$, then $h$ can be expressed as a
product of five conjugates of $g$ or $g^{-1}$. (This clearly implies
simplicity.) Recently Macpherson and Tent \cite{ch32:new11} gave a
different proof of simplicity which applies in more general situations. 

Truss also described the cycle structures of all elements of 
$\Aut(R)$.

A countable structure $M$ is said to have the \emph{small index
property} if any subgroup of $\Aut(M)$ with index less than
$2^{\aleph_0}$ contains the pointwise stabilizer of a finite set of
points of $M$i; it has the \emph{strong small index property} if any
such subgroup lies between the pointwise and setwise stabilizer of a finite
set. Hodges {\it et al.} \cite{ch32:bib29} and Cameron \cite{ch32:new4}
showed:

\begin{theorem}\label{ch32:them8.2} 
$R$ has the strong small index property.
\end{theorem}

The significance of this appears in the next subsection.
It is also related to the question of the reconstruction of a
structure from its automorphism group. For example,
Theorem~\ref{ch32:them8.2} has the following consequence:

\begin{corollary}\label{ch32:coro8.1}
Let $\Gamma$ be a graph with fewer than $2^{\aleph_0}$
vertices, on which $\Aut(R)$ acts transitively on vertices,
edges and non-edges. Then $\Gamma$ is isomorphic to $R$ (and the
isomorphism respects the action of $\Aut(R)$).
\end{corollary}

\subsection{Topology}

The symmetric group $\Sym(X)$ on an infinite set $X$ has a natural topology,
in which a neighbourhood basis of the identity is given by the pointwise
stabilizers of finite tuples. In the case where $X$ is countable, this
topology is derived from a complete metric, as follows. Take
$X=\mathbb{N}$.

Let $m(g)$ be the smallest point moved by the permutation $g$. Take the
distance between the identity and $g$ to be $\max\{2^{-m(g)},2^{-m(g^{-1})}\}$.
Finally, the metric is translation-invariant, so that $d(f,g)=d(fg^{-1},1)$.

\begin{proposition}
Let $G$ be a subgroup of the symmetric group on a countable set $X$. Then the
following are equivalent:
\begin{itemize}
\item[(a)] $G$ is closed in $\Sym(X)$;
\item[(b)] $G$ is the automorphism group of a first-order structure on $X$;
\item[(c)] $G$ is the automorphism group of a homogeneous relational structure
on $X$.
\end{itemize}
\end{proposition}

So automorphism groups of homogeneous relational structures such as $R$ are
themselves topological groups whose topology is derived from a complete metric.

In particular, the Baire category theorem applies to groups like $\Aut(R)$.
So we can ask: is there a
``typical'' automorphism? Truss \cite{ch32:bib53} showed the
following result.

\begin{theorem}\label{ch32:them9.2}
There is a conjugacy class which is residual in
$\Aut(R)$. Its members have infinitely many cycles of each
finite length, and no infinite cycles.
\end{theorem}

Members of the residual conjugacy class (which is, of course,
unique) are called \emph{generic automorphisms} of $R$. I outline
the argument. Each of the following sets of automorphisms is
residual:
\begin{itemize}
\item[(a)] those with no infinite cycles;
\item[(b)] those automorphisms $g$ with the property that, 
if $\Gamma$ is any finite graph and $f$
any isomorphism between subgraphs of $\Gamma$, then there is an
embedding of $\Gamma$ into $R$ in such a way that $g$ extends $f$.
\end{itemize}
(Here (a) holds because the set of automorphisms for which the first
$n$ points lie in finite cycles is open and dense.) In fact, (b) can
be strengthened; we can require that, if the pair $(\Gamma, f)$
extends the pair $(\Gamma_0, f_0)$ (in the obvious sense), then any
embedding of $\Gamma_0$ into $R$ such that $g$ extends $f_0$ can be
extended to an embedding of $\Gamma$ such that $g$ extends $f$. Then
a residual set of automorphisms satisfy both (a) and the
strengthened (b); this is the required conjugacy class.

Another way of expressing this result is to consider the class
$\mathcal{C}$ of finite structures each of which is a graph $\Gamma$
with an isomorphism $f$ between two induced subgraphs (regarded as a
binary relation). This class satisfies Fra\"{i}ss\'{e}'s hypotheses,
and so has a Fra\"{i}ss\'{e} limit $M$. It is not hard to show that,
as a graph, $M$ is the random graph $R$; arguing as above, the map
$f$ can be shown to be a (generic) automorphism of $R$.

More generally, Hodges \textit{et al.} \cite{ch32:bib29} showed that
there exist ``generic $n$-tuples'' of automorphisms of $R$, and used
this to prove the small index property for $R$; see also Hrushovski
\cite{ch32:bib30}. The group generated by a generic $n$-tuple of
automorphisms is, not surprisingly, a free group; all its orbits are
finite. In the next subsection, we turn to some very different subgroups.

To conclude this section, we revisit the strong small index property. Recall
that a neighbourhood basis for the identity consists of the pointwise 
stabilisers of finite sets. If the strong small index property holds, then
every subgroup of small index (less than $2^{\aleph_0}$) contains one of
these, and so is open. So we can take the subgroups of small index as a
neighbourhood basis of the identity. So we have the following reconstruction
result:

\begin{proposition}
\label{ch32:reconst}
If $M$ is a countable structure with the strong small index property (for
example, $R$), then the structure of $\Aut(M)$ as topological group is
determined by its abstract group structure.
\end{proposition}

\subsection{Subgroups}

Another field of study concerns small subgroups. To introduce this,
we reinterpret the last construction of $R$ in
Section~\ref{ch32:sec2.2}. Recall that we took a universal set $S
\subseteq \mathbb{N}$, and showed that the graph $\Gamma(S)$ with
vertex set $\mathbb{Z}$, in which $x$ and $y$ are adjacent whenever
$|x- y|\in S$, is isomorphic to $R$. Now this graph admits the
``shift'' automorphism $x \mapsto  x + 1$, which permutes the
vertices in a single cycle. Conversely, let $g$ be a cyclic
automorphism of $R$. We can index the vertices of $R$ by integers so
that $g$ is the map $x \mapsto  x + 1$. Then, if $S = \{n \in
\mathbb{N}: n \sim 0\}$, we see that $x \sim y$ if and only if $|x-
y|\in S$, and that $S$ is universal. A short calculation shows that
two cyclic automorphisms are conjugate in $\Aut(R)$ if and only
if they give rise to the same set $S$. Since there are
$2^{\aleph_0}$ universal sets, we conclude:

\begin{proposition}\label{ch32:prop8.2} 
$R$ has $2^{\aleph_0}$ non-conjugate cyclic automorphisms.
\end{proposition}

(Note that this gives another proof of Proposition~\ref{ch32:prop8.1}.)

Almost all subsets of $\mathbb{N}$ are universal --- this is true in
either sense discussed in Section~\ref{ch32:sec2.7}. The
construction preceding Proposition~\ref{ch32:prop8.2} shows that
graphs admitting a given cyclic automorphism correspond to subsets
of $\mathbb{N}$; so almost all ``cyclic graphs'' are isomorphic to
$R$. What if the cyclic permutation is replaced by an arbitrary
permutation or permutation group? The general answer is unknown:

\begin{conjecture}\label{ch32:conj8.1} 
Given a permutation group $G$ on a countable set, the following are equivalent:
\begin{itemize}
\item[(a)] some $G$-invariant graph is isomorphic to $R$;
\item[(b)] a random $G$-invariant graph is isomorphic to $R$ with positive probability.
\end{itemize}
\end{conjecture}

A random $G$-invariant graph is obtained by listing the orbits of
$G$ on the $2$-subsets of the vertex set, and deciding randomly
whether the pairs in each orbit are edges or not. We cannot replace
``positive probability'' by ``probability $1$'' here. For example,
consider a permutation with one fixed point $x$ and two infinite
cycles. With probability $\frac{1}{2}$, $x$ is joined to all or none
of the other vertices; if this occurs, the graph is not isomorphic
to $R$. However, almost all graphs for which this event does not
occur are isomorphic to $R$. It can be shown that the conjecture is
true for the group generated by a single permutation; and Truss'
list of cycle structures of automorphisms can be re-derived in this
way.

Another interesting class consists of the \emph{regular}
permutation groups. A group is \emph{regular} if it is transitive
and the stabilizer of a point is the identity. Such a group $G$ can
be considered to act on itself by right multiplication. Then any
$G$-invariant graph is a \emph{Cayley graph} for $G$; in other
words, there is a subset $S$ of $G$, closed under inverses and not
containing the identity, so that $x$ and $y$ are adjacent if and
only if $xy^{-1}\in S$. Now we can choose a \emph{random Cayley
graph} for $G$ by putting inverse pairs into $S$ with probability
$\frac{1}{2}$. It is not true that, for every countable group $G$, a
random Cayley graph for $G$ is almost surely isomorphic to $R$.
Necessary and sufficient conditions can be given; they are somewhat
untidy. I will state here a fairly general sufficient condition.

A \emph{square-root set} in $G$ is a set
\[
\sqrt{a} = \{x \in G : x^2=a\};
\]
it is \emph{principal} if $a = 1$, and \emph{non-principal}
otherwise.

\begin{proposition}\label{ch32:prop8.3} 
Suppose that the countable group $G$ cannot be
expressed as the union of finitely many translates of non-principal
square-root sets and a finite set. Then almost all Cayley graphs for
$G$ are isomorphic to $R$.
\end{proposition}

This proposition is true in the sense of Baire category as well. In
the infinite cyclic group, a square-root set has cardinality at most
$1$; so the earlier result about cyclic automorphisms follows. See
Cameron and Johnson \cite{ch32:bib9} for further details.

\subsection{Overgroups}

There are a number of interesting overgroups of $\Aut(R)$ in the symmetric
group on the vertex set $X$ of $R$.

Pride of place goes to the \emph{reducts}, the overgroups which are closed
in the topology on $\Sym(X)$ (that is, which are automorphism groups of
relational structures which can be defined from $R$ without parameters).
These were classified by Simon Thomas \cite{ch32:bib50}.

An \emph{anti-automorphism}
of $R$ is an isomorphism from $R$ to its complement; a
\emph{switching automorphism} maps $R$ to a graph equivalent to
$R$ by switching. The concept of a \emph{switching
anti-automorphism} should be clear.

\begin{theorem}\label{ch32:them8.3} 
There are exactly five reducts of $R$, viz.: $A = \Aut(R)$; the
group $D$ of automorphisms and anti-automorphisms of $R$; the group
$S$ of switching automorphisms of $R$; the group $B$ of switching
automorphisms and anti-automorphisms of $R$; and the symmetric
group.
\end{theorem}

\begin{remark}\label{ch32:rema8.1}\rm
The set of all graphs on a given vertex set is
a $\mathbb{Z}_2$-vector space, where the sum of two graphs is
obtained by taking the symmetric difference of their edge sets. Now
complementation corresponds to adding the complete graph, and
switching to adding a complete bipartite graph. Thus, it follows
from Theorem~\ref{ch32:them8.3} that, if $G$ is a closed supergroup
of $\Aut(R)$, then the set of all images of $R$ under $G$ is
contained in a coset of a subspace $W(G)$ of this vector space. (For
example, $W(B)$ consists of all complete bipartite graphs and all
unions of at most two complete graphs.) Moreover, these subspaces
are invariant under the symmetric group. It is remarkable that the
combinatorial proof leads to this algebraic conclusion.
\end{remark}

Here is an application due to Cameron and Martins
\cite{ch32:bib10}, which draws together several threads from earlier
sections. Though it is a result about finite random graphs, the
graph $R$ is inextricably involved in the proof.

Let $\mathcal{F}$ be a finite collection of finite graphs. For any
graph $\Gamma$, let $\mathcal{F}(\Gamma)$ be the hypergraph whose
vertices are those of $\Gamma$, and whose edges are the subsets
which induce graphs in $\mathcal{F}$. To what extent does
$\mathcal{F}(\Gamma)$ determine $\Gamma$?

\begin{theorem}\label{ch32:them8.4} 
Given $\mathcal{F}$, one of the following
possibilities holds for almost all finite random graphs $\Gamma$:
\begin{itemize}
\item[(a)] $\mathcal{F}(\Gamma)$ determines $\Gamma$ uniquely;
\item[(b)] $\mathcal{F}(\Gamma)$ determines $\Gamma$ up to complementation;
\item[(c)] $\mathcal{F}(\Gamma)$ determines $\Gamma$ up to switching;
\item[(d)] $\mathcal{F}(\Gamma)$ determines $\Gamma$ up to switching and/or complementation;
\item[(e)] $\mathcal{F}(\Gamma)$ determines only the number of vertices of $\Gamma$.
\end{itemize}
\end{theorem}

I sketch the proof in the first case, that in which $\mathcal{F}$ is
not closed under either complementation or switching. We distinguish
two first-order languages, that of graphs and that of hypergraphs
(with relations of the arities appropriate for the graphs in
$\mathcal{F}$). Any sentence in the hypergraph language can be
``translated'' into the graph language, by replacing ``$E$ is an
edge'' by ``the induced subgraph on $E$ is one of the graphs in
$\mathcal{F}$''.

By the case assumption and Theorem~\ref{ch32:them8.3}, we have
$\Aut(\mathcal{F}(R)) = \Aut(R)$. Now by
Theorem~\ref{ch32:them5.2}, the edges and non-edges in $R$ are
$2$-types in $\mathcal{F}(R)$, so there is a formula $\phi(x, y)$ (in
the hypergraph language) such that $x \sim y$ in $R$ if and only if
$\phi(x, y)$ holds in $\mathcal{F}(R)$. If $\phi^{\ast}$ is the
``translation'' of $\phi$, then $R$ satisfies the sentence
\[
(\forall x, y)((x \sim y) \leftrightarrow \phi^{\ast}(x, y)).
\]
By Theorem~\ref{ch32:them6.1}, this sentence holds in almost all
finite graphs. Thus, in almost all finite graphs, $\Gamma$, vertices
$x$ and $y$ are joined if and only if $\phi(x, y)$ holds in
$\mathcal{F}(\Gamma)$. So $\mathcal{F}(\Gamma)$ determines $\Gamma$
uniquely.

By Theorem~\ref{ch32:them8.3}, $\Aut(\mathcal{F}(R))$ must be
one of the five possibilities listed; in each case, an argument like
the one just given shows that the appropriate conclusion holds.

\medskip

There are many interesting overgroups of $\Aut(R)$ which are not closed,
some of which are surveyed (and their inclusions determined) in a
forthcoming paper of Cameron \emph{et al.} \cite{ch32:new5}. These
arise in one of two ways.

First, we can take automorphism groups of non-relational structures, such as
hypergraphs with infinite hyperedges (for example, take the hyperedges to
be the subsets of the vertex set which induce subgraphs isomorphic to $R$),
or topologies or filters (discussed in the next section). Second, we may
weaken the notion of automorphism. For example, we have a chain of subgroups
\[
\Aut(R)<\Aut_1(R)<\Aut_2(R)<\Aut_3(R)<\Sym(V(R))
\]
with all inclusions proper, where
\begin{itemize}
\item $\Aut_1(R)$ is the set of permutations which change only finitely many
adjacencies (such permutations are called \emph{almost automorphisms} of $R$);
\item $\Aut_2(R)$ is the set of permutations which change only finitely many
adjacencies at any vertex of $R$;
\item $\Aut_3(R)$ is the set of permutations which change only finitely many
adjacencies at all but finitely many vertices of $R$.
\end{itemize}

All these groups are \emph{highly transitive}, that is, given any two
$n$-tuples $(v_1,\ldots,v_n)$ and $(w_1,\ldots,w_n)$ of distinct vertices,
there is an element of the relevant group carrying the first tuple to the
second. This follows from $\Aut_1(R)$ by the indestructibility of $R$. If
$R_1$ and $R_2$ are the graphs obtained by deleting all edges within
$\{v_1,\ldots,v_n\}$ and within $\{w_1,\ldots,w_n\}$ respectively, then
$R_1$ and $R_2$ are both isomorphic to $R$. By homogeneity of $R$, there
is an isomorphism from $R_1$ to $R_2$ mapping $(v_1,\ldots,v_n)$ to
$(w_1,\ldots,w_n)$; clearly this map is an almost-automorphism of $R$.

Indeed, any overgroup of $R$ which is not a reduct preserves no non-trivial
relational structure, and so must be highly transitive.

\section{Topological aspects}
\label{ch32:sec2.9}

There is a natural way to define a topology on the vertex set of
$R$: we take as a basis for the open sets the set of all finite
intersections of vertex neighbourhoods. It can be shown that this
topology is homeomorphic to $\mathbb{Q}$ (using the characterization
of $\mathbb{Q}$ as the unique countable, totally disconnected,
topological space without isolated points, due to Sierpi\'{n}iski
\cite{ch32:bib48}, see also Neumann \cite{ch32:bib38}). Thus:

\begin{proposition}\label{ch32:prop9.1} 
$\Aut(R)$ is a subgroup of the homeomorphism group of $\mathbb{Q}$.
\end{proposition}

This is related to a theorem of Mekler \cite{ch32:bib36}:

\begin{theorem}\label{ch32:them9.1} 
A countable permutation group $G$ is embeddable
in the homeomorphism group of $\mathbb{Q}$ if and only if the
intersection of the supports of any finite number of elements of $G$
is empty or infinite.
\end{theorem}

Here, the support of a permutation is the set of points it doesn't
fix. Now of course $\Aut(R)$ is not countable; yet it does
satisfy Mekler's condition. (If $x$ is moved by each of the
automorphisms $g_1,\ldots,g_n$, then the infinitely many vertices
joined to $x$ but to none of $xg_1,\ldots, xg_n$ are also moved by
these permutations.)

The embedding in Proposition \ref{ch32:prop9.1} can be realised constructively:
the topology can be defined directly from the graph. Take a basis for the
open sets to be the sets of witnesses for our defining property $(*)$; that is, 
sets of the form
\[Z(U,V)=\{z\in V(R): (\forall u\in U)(z\sim u) \wedge (\forall v\in V)(z \not \sim v)\}\]
for finite disjoint sets $U$ and $V$. Now given $u\ne v$, there is a point
$z\in Z(\{u\},\{v\})$; so the open neighbourhood of $z$ is open and closed in
the topology and contains $u$ but not $v$. So the topology is totally
disconnected. It has no isolated points, so it is homeomorphic to $\mathbb{Q}$,
by Sierpi\'nski's Theoreom.

There is another interesting topology on the vertex set of $R$, which can be
defined in three different ways. Let $B$ be the ``random bipartite graph'',
the graph with vertex set $X\cup Y$ where $X$ and $Y$ are countable and
disjoint, where edges between $X$ and $Y$ are chosen randomly. (A simple
modification of the Erd\H{o}s--R\'enyi argument shows that there is a unique
graph which occurs with probability $1$.) Now consider the following
topologies on a countable set $X$:
\begin{description}
\item{$\mathcal{T}$:} point set $V(R)$, sub-basic open sets are
open vertex neighbourhoods.
\item{$\mathcal{T}^*$:} points set $V(R)$, sub-basic open sets are
closed vertex neighbourhoods.
\item{$\mathcal{T}^\dag$:} points are one bipartite block in $B$, sub-basic
open sets are neighbourhoods of vertices in the other bipartite block.
\end{description}

\begin{proposition}
\begin{itemize}
\item[(a)] The three topologies defined above are all homeomorphic.
\item[(b)] The homeomorphism groups of these topologies are highly transitive.
\end{itemize}
\end{proposition}

Note that the topologies are homeomorphic but not identical. For example,
the identity map is a continuous bijection from $\mathcal{T}^*$ to
$\mathcal{T}$, but is not a homeomorphism.

\section{Some other structures}
\label{ch32:sec2.10}

\subsection{General results}

As we have seen, $R$ has several properties of a general kind: for
example, homogeneity, $\aleph_0$-categoricity, universality,
ubiquity. Much effort has gone into studying, and if possible
characterizing, structures of other kinds with these properties.
(For example, they are all shared by the ordered set $\mathbb{Q}$.)

Note that, of the four properties listed, the first two each imply
the third, and the first implies the fourth. Moreover, a homogeneous
structure over a finite relational language is
$\aleph_0$-categorical, since there are only finitely many
isomorphism types of $n$-element structure for each $n$. Thus,
homogeneity is in practice the strongest condition, most likely to
lead to characterizations.

A major result of Lachlan and Woodrow \cite{ch32:bib35} determines
the countable homogeneous graphs. The graphs $H_n$ in this theorem are
so-called because they were first constructed by Henson \cite{ch32:bib26}.

\begin{theorem}\label{ch32:them10.1} 
A countable homogeneous graph is isomorphic to one of the following:
\begin{itemize}
\item[(a)] the disjoining union of $m$ complete graphs of size $n$, where $m, n \leq \aleph_0$ and at
least one of $m$ and $n$ is $\aleph_0$;
\item[(b)] complements of (a);
\item[(c)] the Fra\"{i}ss\'{e} limit $H_n$ of the class of $K_n$-free graphs, for fixed $n \geq 3$;
\item[(d)] complements of (c);
\item[(e)] the random graph $R$.
\end{itemize}
\end{theorem}

The result of Macpherson and Tent \cite{ch32:new11} shows that the
automorphism groups of the Henson graphs are simple.
It follows from Proposition \ref{ch32:reconst} that $\Aut(R)$ is not
isomorphic to $\Aut(H_n)$. It is
not known whether these groups are pairwise non-isomorphic.

Other classes in which the homogeneous structures have been
determined include finite graphs (Gardiner \cite{ch32:bib23}),
tournaments (Lachlan \cite{ch32:bib34} --- surprisingly, there are
just three), digraphs (Cherlin \cite{ch32:bib12} (there are 
uncountably many, see Henson \cite{ch32:bib27}), and posets (Schmerl
\cite{ch32:bib45}). In the case of posets, Droste \cite{ch32:bib14} has
characterizations under weaker assumptions.

For a number of structures, properties of the automorphism group,
such as normal subgroups, small index property, or existence of
generic automorphisms, have been established.

A theorem of Cameron \cite{ch32:bib4} determines the reducts of
$\Aut(\mathbb{Q})$:

\begin{theorem}\label{ch32:them10.2} 
There are just five closed permutation groups
containing the group $\Aut(\mathbb{Q})$ of order-preserving
permutations of $\mathbb{Q}$, viz.: $\Aut(\mathbb{Q})$; the group of
order preserving or reversing permutations; the group of
permutations preserving a cyclic order; the group of permutations
preserving or reversing a cyclic order; and $\Sym(\mathbb{Q})$.
\end{theorem}

However, there is no analogue of Theorem~\ref{ch32:them8.4} in this
case, since there is no Glebskii--B1ass--Fagin--Harary theory for
ordered sets. ($\mathbb{Q}$ is dense; this is a first-order property,
but no finite ordered set is dense.)

Simon Thomas \cite{ch32:new15} has determined the reducts of the random
$k$-uniform hypergraph for all $k$.

Since my paper with Paul Erd\H{o}s concerns sum-free sets (Cameron
and Erd\H{o}s \cite{ch32:bib8}), it is appropriate to discuss their
relevance here. Let $H_n$ be the Fra\"{i}ss\'{e} limit of the class
of $K_n$-free graphs, for $n \geq 3$ (see
Theorem~\ref{ch32:them10.1}). These graphs were first constructed
by Henson \cite{ch32:bib26}, who also showed that $H_3$ admits
cyclic automorphisms but $H_n$ does not for $n > 3$. We have seen
how a subset $S$ of $\mathbb{N}$ gives rise to a graph $\Gamma(S)$
admitting a cyclic automorphism: the vertex set is $\mathbb{Z}$, and
$x \sim y$ if and only if $|x - y|\in S$. Now $\Gamma(S)$ is
triangle-free if and only if $S$ is \emph{sum-free} (i.e., $x,y
\in S \Rightarrow x + y \notin S$). It can be shown that, for almost
all sum-free sets $S$ (in the sense of Baire category), the graph
$\Gamma(S)$ is isomorphic to $H_3$; so $H_3$ has $2^{\aleph_0}$
non-conjugate cyclic automorphisms. However, the analogue of this
statement for measure is false; and, indeed, random sum-free sets
have a rich and surprising structure which is not well understood
(Cameron \cite{ch32:bib5}). For example, the probability that
$\Gamma(S)$ is bipartite is approximately $0.218$. It is conjectured
that a random sum-free set $S$ almost never satisfies $\Gamma(S)\cong H_3$.
In this direction, Schoen \cite{ch32:new14} has shown that, if
$\Gamma(S)\cong H_3$, then $S$ has density zero.

The Henson $K_n$-free graphs $H_n$, being homogeneous, are ubiquitous in the 
sense of Baire category: for example, the set of graphs isomorphic to $H_3$
is residual in the set of triangle-free graphs on a given countable vertex
set (so $H_3$ is ubiquitous, in the sense defined earlier). However, until
recently, no measure-theoretic analogue was known. We saw after Proposition
\ref{ch32:prop7.1} that a random triangle-free graph is almost surely
bipartite! However, Petrov and Vershik \cite{ch32:new13} recently
managed to construct
an exchangeable measure on graphs on a given countable vertex set which is
concentrated on Henson's graph. More recently, Ackerman, Freer and Patel
showed that the construction works much more generally: the necessary and
sufficient condition turns out to be the strong amalgamation property, which
we discussed in Section \ref{ch32:sec2.5}.

Universality of a structure $M$ was defined in a somewhat
introverted way in Section~\ref{ch32:sec2.5}: $M$ is universal if
every structure younger than $M$ is embeddable in $M$. A more
general definition would start with a class $\mathcal{C}$ of
structures, and say that $M \in \mathcal{C}$ is \emph{universal}
for $\mathcal{C}$ if every member of $\mathcal{C}$ embeds into $M$.
For a survey on this sort of universality, for
various classes of graphs, see Komjath and Pach \cite{ch32:bib33}.
Two negative results, for the classes of locally finite graphs and
of planar graphs, are due to De Bruijn (see Rado \cite{ch32:bib42})
and Pach \cite{ch32:bib40} respectively.

\subsection{The Urysohn space}

A remarkable example of a homogeneous structure is the celebrated
Urysohn space, whose construction predates Fra\"{\i}ss\'e's work by
more than two decades. Urysohn's paper \cite{ch32:new16} was published
posthumously, following his drowning in the Bay of Biscay at the age of
26 on his first visit to western Europe (one of the most romantic stories
in mathematics). An exposition of the Urysohn space is given by 
Vershik \cite{ch32:new17}.

The Urysohn space is a complete separable metric space $\mathbb{U}$ which
is universal (every finite metric space is isometrically embeddable in
$\mathbb{U}$) and homogeneous (any isometry between finite subsets can be
extended to an isometry of the whole space). Since $\mathbb{U}$ is
uncountable, it is not strictly covered by the Fra\"{\i}ss\'e theory, but
one can proceed as follows. The set of finite \emph{rational} metric
spaces (those with all distances rational) is a Fra\"{\i}ss\'e class;
the restriction to countable distances ensures that there are only countably
many non-isomorphic members. Its Fra\"{\i}ss\'e limit is the so-called
\emph{rational Urysohn space} $\mathbb{U}_Q$. Now the Urysohn space is the
completion of $\mathbb{U}_Q$.

Other interesting homogeneous metric spaces can be constructed similarly,
by restricting the values of the metric in the finite spaces. For example,
we can take integral distances, and obtain the \emph{integral Urysohn space}
$\mathbb{U}_Z$. We can also take distances from the set $\{0,1,2,\ldots,k\}$
and obtain a countable homogeneous metric space with these distances. For
$k=2$, we obtain precisely the path metric of the random graph $R$. (Property
$(\ast)$ guarantees that, given two points at distance $2$, there is a point
at distance $1$ from both; so, defining two points to be adjacent if they are
at distance $1$, we obtain a graph whose path metric is the given metric. It
is easily seen that this graph is isomorphic to $R$.)

Note that $R$ occurs
in many different ways as a reduct of $\mathbb{U}_Q$. Split the positive
rationals into two dense subsets $A$ and $B$, and let two points $v,w$ be
adjacent if $d(v,w)\in A$; the graph we obtain is $R$.

A study of the isometry group of the Urysohn space, similar to that done for
$R$, was given by Cameron and Vershik \cite{ch32:new7}. The automorphism
group is not simple, since the isometries which move every point by a bounded
distance form a non-trivial normal subgroup.

\subsection{KPT theory}

I conclude with a brief discussion of a dramatic development at the interface
of homogeneous structures, Ramsey theory, and topological dynamics.

The first intimation of such a connection was pointed out by 
Ne\v{s}et\v{r}il \cite{ch32:new10}, and in detail in Hubi\v{c}ka and
Ne\v{s}et\v{r}il \cite{ch32:new9}. We use the notation $A\choose B$ for the
set of all substructures of $A$ isomorphic to $B$. A class $\mathcal{C}$ of
finite structures is a \emph{Ramsey class} if, given a natural number $r$ and
a pair $A,B$ of structures in $\mathcal{C}$, there exists a structure
$C\in\mathcal{C}$ such that, if $C\choose A$ is partitioned into $r$
classes, then there is an element $B'\in{C\choose B}$ for which $B'\choose A$
is contained in a single class. In other words, if we colour the
$A$-substructures of $C$ with $r$ colours, then there is a $B$-substructure
of $C$, all of whose $A$-substructures belong to the same class. The
classical theorem of Ramsey asserts that the class of finite sets is a
Ramsey class.

\begin{theorem}
A hereditary, isomorphism-closed Ramsey class is a Fra\"{\i}ss\'e class.
\end{theorem}

There are simple examples which show that a good theory of Ramsey classes
can only be obtained by making the objects rigid. The simplest way to do this
is to require that a total order is part of the structure. Note that, if a
Fra\"{\i}ss\'e class has the strong amalgamation property, than we may adjoin
to it a total order (independent of the rest of the structure) to obtain a
new Ramsey class. We refer to \emph{ordered structures} in this situation.
Now the theorem above suggests a procedure for finding Ramsey classes: take
a Fra\"{\i}ss\'e class of ordered structures and test the Ramsey property.
A number of Ramsey classes, old and new, arise in this way: ordered graphs,
$K_n$-free graphs, metric spaces, etc. Indeed, if we take an ordered set
and ``order'' it as above to obtain a set with two orderings, we obtain
the class of \emph{permutation patterns}, which is also a Ramsey class:
see Cameron \cite{ch32:new2}, B\"ottcher and Foniok \cite{ch32:new4}.

The third vertex of the triangle was quite unexpected. 

A \emph{flow} is a continuous action of a topological group $G$ on a 
topological space $X$, usually assumed to be a compact Hausdorff space.
A topological group $G$ admits a unique \emph{minimal flow}, or universal
minimal continuous action on a compact space $X$. (Here \emph{minimal} means
that $X$ has no non-empty proper closed $G$-invariant subspace, and
\emph{universal} means that it can be mapped onto any minimal $G$-flow.)

The group $G$ is said to be \emph{extremely amenable} if its minimal flow
consists of a single point.

The theorem of Kechris, Pestov and Todorcevic \cite{ch32:new10} asserts:

\begin{theorem}
Let $X$ be a countable set, and $G$ a closed subgroup of $\Sym(X)$. Then
$G$ is extremely amenable if and only if it is the automorphism group of
a homogeneous structure whose age is a Ramsey class of ordered structures.
\end{theorem}

As a simple example, the theorem shows that $\Aut(\mathbb{Q})$ (the group
of order-preserving permutations of $\mathbb{Q}$ is extremely amenable
(a result of Pestov).

The fact that the two conditions are equivalent allows information to be
transferred in both directions between combinatorics and topological dynamics.
In particular, known Ramsey classes such as ordered graphs, ordered $K_n$-free
graphs, ordered metric spaces, and permutation patterns give examples of
extremely amenable groups.

The theorem can also be used in determining the minimal flows for various
closed subgroups of $\Sym(X)$. For example, the minimal flow for $\Sym(X)$ is
the set of all total orderings of $X$ (a result of Glasner and Weiss
\cite{ch32:new8}).

\end{document}